\documentclass[11pt]{article}

\usepackage[T1]{fontenc}
\usepackage{lmodern}
\usepackage{amsmath,amssymb,amsthm}
\usepackage[a4paper,margin=1in]{geometry}
\usepackage[hidelinks]{hyperref}
\usepackage{microtype}

\newtheorem{proposition}{Proposition}
\theoremstyle{remark}

\title{Direct Generation of a Somos--4 Sequence\\
from an Algebraic Generating Function}
\author{Thomas Scheuerle}
\date{}

\begin{document}
\maketitle

\begin{abstract}
Let
\[
G(z)=\sum_{n\ge 0} g_n z^n
\]
be a formal power series, and let
\[
H_n=\det(g_{i+j})_{0\le i,j\le n-1},\qquad H_0=1.
\]
We consider the five-parameter family in which
\[
H_1=\lambda,\qquad H_2=m,\qquad H_3=r,\qquad H_4=\eta,
\]
and the Hankel sequence satisfies the Somos--4 recurrence
\[
H_{n+4}H_n=H_{n+3}H_{n+1}+\tau H_{n+2}^2.
\]
We construct a quadratic algebraic generating function whose coefficient
sequence has these initial Hankel determinants and whose Hankel transform lies
in the family $A(1,\tau)$.  The construction begins with a Stieltjes
continued fraction, passes to a normalized form, and determines the
coefficients by a paired-Hankel reconstruction.  The resulting algebraic
equation is then obtained by eliminating the continued-fraction tail.  The
two algebraic branches coalesce at the origin, and the desired branch is
selected by the initial conditions. The essential point is that the construction removes this ambiguity by
prescribing, in addition to the ordinary Hankel transform, a companion Hankel
sequence obtained from the same Somos--4 orbit by a two-step shift.
\end{abstract}

\section*{Acknowledgments}
I am grateful to Florian Lengyel for reading an earlier version of this manuscript and for pointing out important errors, and to Michael Somos for pointing out important suggestions for improvement.

\section{Context and relation to earlier work}
The present note is concerned with an inverse Hankel-realization problem.  We
start from prescribed initial determinants
\[
(H_1,H_2,H_3,H_4)=(\lambda,m,r,\eta)
\]
and a Somos--4 recurrence with parameter $\tau$, and we construct a formal
power series whose leading Hankel minors realize the corresponding
$A(1,\tau)$ orbit.  In addition, the construction fixes the once-shifted
minors, and this paired-Hankel normalization determines the odd and even
Stieltjes coefficients separately.  The resulting alternating tail recurrence
closes quadratically, and after restoring the initial levels one obtains an
explicit algebraic equation for the generating function.

This is the inverse counterpart to the better-known forward problem: given a
generating function, determine whether its Hankel transform satisfies a Somos
recurrence.  The latter has been studied, for example, by Xin, Barry,
Wang--Zhang, and Liu--Wang--Zhang.  The present construction takes the
opposite direction and seeks a direct realization of prescribed Somos--4
Hankel data.

Seeking literature on the inverse problem, we find:  Chang, Hu, and Xin
construct Hankel solutions after a gauge normalization of adjacent
terms~\cite{ChangHuXin2015}.  Hone develops genus-one continued-fraction and
Hankel formulae, including gauge factors and an additional moment
parameter~\cite{Hone2021}.  Chang and Liu treat generic bilateral initial
conditions for quadratic orthogonal pairs and obtain algebraic series,
coefficient recurrences, and Hankel representations~\cite{ChangLiu2026}.
These works provide the closest setting for the present construction.

The contribution of this note is a direct Stieltjes-fraction realization of
generic $A(1,\tau)$ data in a particular positive-order paired-Hankel
normalization.  Here $H_0=1$ is the empty-determinant convention, while the
recurrence starts at $n=1$, so the four initial determinants remain
independent.  The ordinary and once-shifted determinant prescriptions recover
the Stieltjes coefficients, and their alternating recurrence is compatible
with a quadratic family of continued-fraction tails.  Restoring the initial
levels then gives the explicit quadratic equation recorded below.

The underlying tools come from the classical theory of continued fractions,
moment sequences, and Hankel determinants; see Wall~\cite{Wall1948}, Jones and
Thron~\cite{JonesThron1980}, Flajolet~\cite{Flajolet1980},
Krattenthaler~\cite{Krattenthaler1999,Krattenthaler2005},
Layman~\cite{Layman2001}, and French~\cite{French2007}.  The quadratic
transformation of Sulanke and Xin~\cite{SulankeXin2007} and the related
methods of Gessel and Xin~\cite{GesselXin2006} are particularly effective for
structured Hankel identities.  The wider Somos context includes the Laurent
phenomenon and elliptic dynamics
\cite{Gale1991,FominZelevinsky2002,VanderPoortenSwart2006,Hone2005,HoneSwart2008}.
\section{Hankel determinants and Somos parameters}

Let
\[
 G(z)=\sum_{n\geq0}g_nz^n
\]
be a formal power series, and for $n\geq1$ define its ordinary Hankel
determinants by
\begin{equation}
 H_n=\det(g_{i+j})_{0\leq i,j\leq n-1},
 \qquad H_0=1.
 \label{eq:Hankel}
\end{equation}

The five parameters $\lambda,m,r,\eta,\tau$ have the following roles:
\begin{equation}
 H_1=\lambda,\qquad H_2=m,\qquad H_3=r,\qquad H_4=\eta,
 \label{eq:initial-Hankel}
\end{equation}
and the Hankel sequence satisfies the Somos--4 recurrence
\begin{equation}
 H_{n+4}H_n=H_{n+3}H_{n+1}+\tau H_{n+2}^2,\qquad n\geq1.
 \label{eq:Somos}
\end{equation}
Thus $\tau$ is the second coefficient of the Somos family $A(1,\tau)$,
whereas $\lambda,m,r,$ and $\eta$ select the initial Hankel determinants.
The algebraic generating function producing these determinants is given in
the next section.  The construction makes one additional choice that is easy
to overlook: together with $(H_n)_{n\geq0}$, it fixes the companion
Hankel sequence
\begin{equation}
 \widehat H_n=\det(g_{i+j+1})_{0\leq i,j\leq n-1}
 =H_{n+2}\qquad(n\geq1).
 \label{eq:companion-Hankel-choice}
\end{equation}
Thus the companion is obtained by advancing the Somos--4 orbit by two indices
without applying any additional normalization factor.  Its first term is
$\widehat H_1=g_1=H_3=r$, not $H_2=m$.  This second prescription is part
of the definition of the selected coefficient sequence; it is not a
consequence of the ordinary Hankel determinants alone.  Section~\ref{sec:nonuniqueness} explains why prescribing
both sequences removes the degree of freedom discussed there.

\section{Behind the construction: the Stieltjes fraction}
The quadratic equation below may be obtained by first constructing a
Stieltjes continued fraction whose Hankel determinants follow the prescribed
Somos--4 orbit.  Write
\begin{equation}
 G(x)=\cfrac{c_0}{1-\cfrac{c_1x}{1-\cfrac{c_2x}{1-\cfrac{c_3x}{\ddots}}}}.
 \label{eq:S-fraction}
\end{equation}
For this normalization, the ordinary Hankel determinants are\cite{Wall1948,JonesThron1980,Flajolet1980}
\begin{equation}
 H_n=c_0^n\prod_{j=1}^{n-1}(c_{2j-1}c_{2j})^{n-j}.
 \label{eq:S-fraction-Hankel}
\end{equation}
Thus
\[
 H_1=c_0,\qquad
 H_2=c_0^2c_1c_2,\qquad
 H_3=c_0^3c_1^2c_2^2c_3c_4,
\]
and similarly at higher orders.

It is useful first to remove the exponential scaling caused by $H_1$.
Set
\begin{equation}
 a_1=1,\qquad
 a_2=\frac{m}{\lambda^2},\qquad
 a_3=\frac{r}{\lambda^3},\qquad
 a_4=\frac{\eta}{\lambda^4}.
 \label{eq:normalized-Somos-data}
\end{equation}
Then $H_n=\lambda^n a_n$ for $1\leq n\leq4$, and the normalized sequence
$(a_n)$ obeys the same Somos--4 recurrence,
\begin{equation}
 a_{n+4}a_n=a_{n+3}a_{n+1}+\tau a_{n+2}^2,\qquad n\geq1.
 \label{eq:normalized-Somos}
\end{equation}
The powers of $\lambda$ in \eqref{eq:normalized-Somos-data} are essential:
for nonunit $H_1=\lambda$, the correct normalization is
$a_j=H_j/\lambda^j$, not $H_j/\lambda^{j-1}$.

A convenient initial segment of the Stieltjes coefficients is
\begin{align}
 c_0&=\lambda, &
 c_1&=\lambda^2a_3, &
 c_2&=\frac{a_2}{\lambda^2a_3},\\
 c_3&=\frac{a_4}{a_2a_3}, &
 c_4&=\frac{a_3^2}{a_2a_4}, &
 c_5&=\tau\frac{a_2a_3}{a_4}+\frac{a_2^2}{a_3}.
 \label{eq:initial-S-coefficients}
\end{align}
Equivalently, in the parameters used throughout this paper,
\begin{equation}
 \begin{aligned}
 c_0&=\lambda,&
 c_1&=\frac{r}{\lambda},&
 c_2&=\frac{m}{\lambda r},\\
 c_3&=\frac{\lambda\eta}{mr},&
 c_4&=\frac{r^2}{m\eta},&
 c_5&=\frac{\tau mr}{\lambda\eta}+\frac{m^2}{\lambda r}.
 \end{aligned}
 \label{eq:initial-S-coefficients-original}
\end{equation}
The remaining coefficients are generated alternately by
\begin{equation}
 c_{2k}=\frac{1}{c_{2k-2}c_{2k-1}},\qquad
 c_{2k+1}=c_{2k-2}+\frac{\tau}{c_{2k-1}}
 \qquad(k\geq3).
 \label{eq:S-coefficient-recurrence}
\end{equation}
In particular,
\[
 c_6=\frac1{c_4c_5},\qquad
 c_7=c_4+\frac{\tau}{c_5},\qquad
 c_8=\frac1{c_6c_7},\qquad
 c_9=c_6+\frac{\tau}{c_7}.
\]
Substitution in \eqref{eq:S-fraction-Hankel} gives
$H_j=\lambda^j a_j$ for $1\leq j\leq4$.  At the next order it gives
\[
 \frac{H_5}{\lambda^5}=\frac{a_3a_4c_5}{a_2}
 =a_2a_4+\tau a_3^2,
\]
which is exactly the first Somos--4 step.  The alternating recurrence
\eqref{eq:S-coefficient-recurrence} continues this mechanism.

The standard once-shifted Hankel product formula~\cite{Krattenthaler2005} for the Stieltjes fraction is
\begin{equation}
 \widehat H_n=(c_0c_1)^n\prod_{j=1}^{n-1}(c_{2j}c_{2j+1})^{n-j}
 \qquad(n\geq1).
 \label{eq:shifted-S-fraction-Hankel}
\end{equation}
Together with \eqref{eq:S-fraction-Hankel}, this formula gives a direct
reconstruction result.

\subsection{Paired-Hankel reconstruction}
Let $F$ be a field, let $H_0=K_0=1$, and suppose $H_nK_n\ne0$ for
$n\geq1$.  There is a unique nondegenerate Stieltjes fraction of the form
\eqref{eq:S-fraction} whose ordinary and once-shifted Hankel determinants
are $H_n$ and $K_n$, respectively.  Its coefficients are
\begin{equation}
 c_0=H_1,\qquad
 c_{2n-1}=\frac{H_{n-1}K_n}{H_nK_{n-1}},\qquad
 c_{2n}=\frac{H_{n+1}K_{n-1}}{H_nK_n}\quad(n\geq1).
 \label{eq:paired-reconstruction}
\end{equation}
Indeed, taking adjacent determinant ratios in
\eqref{eq:S-fraction-Hankel} and
\eqref{eq:shifted-S-fraction-Hankel} and telescoping alternately gives
\eqref{eq:paired-reconstruction}.  Conversely, substitution of these
coefficients into the two product formulae telescopes back to $H_n$ and
$K_n$.  Nonvanishing makes every division legitimate and proves uniqueness
in the S-fraction chart.

In the present construction $K_0=1$ and $K_n=H_{n+2}$ only for $n\geq1$.
Thus $c_1=H_3/H_1$, $c_2=H_2/(H_1H_3)$, and, for $n\geq2$,
\begin{equation}
 c_{2n-1}=\frac{H_{n-1}H_{n+2}}{H_nH_{n+1}},\qquad
 c_{2n}=\frac{H_{n+1}^2}{H_nH_{n+2}}.
 \label{eq:paired-reconstruction-special}
\end{equation}
These identities recover \eqref{eq:initial-S-coefficients-original} and
\eqref{eq:S-coefficient-recurrence}.  They also prove for every positive
order that
\begin{equation}
 \det(g_{i+j+1})_{0\leq i,j\leq n-1}=H_{n+2}\qquad(n\geq1).
 \label{eq:literal-shifted-Hankel}
\end{equation}
In particular, $g_1=c_0c_1=r$.  Thus the ordinary determinants fix only the
products $c_{2j-1}c_{2j}$, whereas the companion determinants fix their
splitting into odd and even coefficients.  This gives the particular
representative used throughout the paper.  The formulas are classical in
substance; their role here is to make the all-index determinant assertions
and the qualified uniqueness statement explicit.  The next step is to eliminate the recursively defined continued-fraction
tail. The following isolated calculation explains why this elimination closes
quadratically and hence provides the conceptual bridge to the explicit equation
in Section~\ref{sec:algebraic-gf}.

\subsection{Quadratic closure of the recurrent tail}
\label{subsec:tail-closure}
For a sequence $(d_n)_{n\geq1}$, write
\begin{equation}
 T_n(x)=\cfrac{1}{1-\cfrac{d_nx}{1-\cfrac{d_{n+1}x}{1-\ddots}}},
 \qquad T_n=\frac{1}{1-d_nxT_{n+1}}.
 \label{eq:tail-fraction}
\end{equation}
The recurrence in \eqref{eq:S-coefficient-recurrence} becomes transparent
after the tail is started at $c_4$. Put
\begin{equation}
 d_n=c_{n+3},\qquad b=d_1=c_4,\qquad d=d_2=c_5,\qquad t=\tau.
 \label{eq:tail-reindexing}
\end{equation}
Then $d_1=b$, $d_2=d$, $d_3=1/(bd)$, and thereafter
\begin{equation}
 d_n=\frac1{d_{n-1}d_{n-2}}\quad(n\geq3\text{ odd}),\qquad
 d_n=\frac{t}{d_{n-2}}+d_{n-3}\quad(n\geq4\text{ even}).
 \label{eq:isolated-tail-recurrence}
\end{equation}
Thus the alternating recurrence can be studied independently of the first
three levels of the original Stieltjes fraction.

\begin{proposition}[Quadratic equation for the recurrent tail]
\label{prop:tail-quadratic}
Let $T=T_1$ be the formal continued fraction defined by
\eqref{eq:tail-fraction}--\eqref{eq:isolated-tail-recurrence}. Then $T(0)=1$ and
\begin{equation}
 \mathcal A(x)T^2+\mathcal B(x)T+\mathcal C(x)=0,
 \label{eq:tail-quadratic}
\end{equation}
where
\begin{align}
 \mathcal A(x)&=(bt+1)x-b\bigl((b+d)t+1\bigr)x^2,
 \label{eq:tail-A}\\
 \mathcal B(x)&=bdt\,x^2+\bigl(-bt+b^2d+d^2b-1\bigr)x-bd,
 \label{eq:tail-B}\\
 \mathcal C(x)&=bd-bd^2x.
 \label{eq:tail-C}
\end{align}
\end{proposition}

\begin{proof}
Removing one level of a Stieltjes fraction is a fractional-linear
transformation,
\begin{equation}
 T_{n+1}=\frac{T_n-1}{d_nxT_n}.
 \label{eq:tail-extraction}
\end{equation}
Consequently, quadratic relations are preserved under tail extraction. More
explicitly, if $Q_n(Y)$ is quadratic and $Q_n(T_n)=0$, then a polynomial for
the next tail is
\begin{equation}
 Q_{n+1}(Y)=(1-d_nxY)^2
 Q_n\!\left(\frac1{1-d_nxY}\right),
 \label{eq:quadratic-transport}
\end{equation}
For the quadratic in \eqref{eq:tail-quadratic}, write
$Q_{b,d}(x,Y)=\mathcal A(x)Y^2+\mathcal B(x)Y+\mathcal C(x)$ and set
\[
 b'=\frac1{bd},\qquad d'=b+\frac{t}{d}.
\]
A direct expansion gives the exact covariance identity
\begin{equation}
 (1-bx-dxZ)^2Q_{b,d}\!\left(x,\frac{1-dxZ}{1-bx-dxZ}\right)
 =\frac{b^3d^4x^2}{bd+t}\,Q_{1/(bd),\,b+t/d}(x,Z).
 \label{eq:two-tail-covariance}
\end{equation}
For the recurrent tails,
\[
 T_1=\frac{1-dxT_3}{1-bx-dxT_3}.
\]
The recurrence \eqref{eq:isolated-tail-recurrence} maps the parameter pair
$(b,d)$ for $T_1$ to $(1/(bd),b+t/d)$ for $T_3$.  Iterating
\eqref{eq:two-tail-covariance} therefore shows, for every $N$, that
$Q_{b,d}(x,T_1)$ is divisible by $x^{2N}$ in
$\mathbb Q(b,d,t)[[x]]$.  Hence it vanishes identically.  Finally,
$Q_{b,d}(0,Y)=bd(1-Y)$, so the formal branch satisfying $T_1(0)=1$ is unique.
\end{proof}
This proposition gives a short derivation of the algebraic nature of the full
generating function. Let $T=T_4$ denote the tail of \eqref{eq:S-fraction}
beginning with $c_4$. The first three levels are recovered by
\begin{equation}
 T_3=\frac1{1-c_3xT},\qquad
 T_2=\frac1{1-c_2xT_3},\qquad
 T_1=\frac1{1-c_1xT_2},\qquad G=c_0T_1.
 \label{eq:recover-full-fraction}
\end{equation}
Each step is again a M\"obius transformation. Substituting
\eqref{eq:recover-full-fraction} into \eqref{eq:tail-quadratic}, followed by
$c_0=\lambda$ and the values in
\eqref{eq:initial-S-coefficients-original}, therefore produces a quadratic
polynomial in $G$. Clearing denominators and collecting powers of $z$ yields
exactly the polynomials in \eqref{eq:A}--\eqref{eq:C}. This explains
structurally why the recursively defined infinite fraction gives a quadratic
algebraic generating function: the recurrence is the compatibility condition
that keeps the continued-fraction tails inside a quadratic family, while the
omitted initial levels only conjugate that family by fractional-linear maps.

\section{The algebraic generating function}
\label{sec:algebraic-gf}

Define the polynomials $A(z),B(z),C(z)$ by
\begin{align}
 A(z)={}&\lambda\eta mr^2(\tau-\eta)z^3 \\
 &+r\Bigl(-(\lambda^2\eta m+m^2(m+r^2))\tau
 +(2m-1)\lambda^2\eta^2+\eta m^2(m+r^2)
 -\lambda r(m+r^2)\Bigr)z^2 \\
 &+\lambda\Bigl(m^2r^2\tau-(m-1)\lambda^2\eta^2
 -m(m^2+2mr^2-m-r^2)\eta+\lambda r^3\Bigr)z \\
 &+\lambda^2\eta m(m-1)r,
 \label{eq:A}\\[0.4em]
 B(z)={}&-\lambda^2\eta mr^2\tau z^3 \\
 &+r\Bigl((2\lambda^3\eta m+\lambda m^2(2m+r^2))\tau
 -(2m-1)\lambda^3\eta^2-\lambda\eta m^3
 +\lambda^2r(2m+r^2)\Bigr)z^2 \\
 &+\lambda^2\Bigl(-2m^2r^2\tau+2(m-1)\lambda^2\eta^2
 +m(2m^2+2mr^2-2m-r^2)\eta-2\lambda r^3\Bigr)z \\
 &-2\lambda^3\eta m(m-1)r,
 \label{eq:B}\\[0.4em]
 C(z)={}&-\lambda^2mr\Bigl(\tau(\lambda^2\eta+m^2)+\lambda r\Bigr)z^2 \\
 &+\Bigl(\lambda^3m^2r^2\tau+\lambda^4r^3
 -\lambda^3\eta(m-1)(\lambda^2\eta+m^2)\Bigr)z \\
 &+\lambda^4\eta m(m-1)r.
 \label{eq:C}
\end{align}
The generating function is the formal-power-series branch of
\begin{equation}
 A(z)G(z)^2+B(z)G(z)+C(z)=0
 \label{eq:quadratic}
\end{equation}
\subsection{Discriminant and residual genus-one curve}
Define
\begin{equation}
 J=\frac{\lambda\eta}{mr}+\frac{m^2}{\lambda r}
   +\frac{r^2}{m\eta}+\frac{\tau mr}{\lambda\eta}.
 \label{eq:J-invariant}
\end{equation}
For the polynomials above, direct expansion and factorization give
\begin{equation}
 B(z)^2-4A(z)C(z)=\bigl(\lambda^2\eta m r^2 z\bigr)^2
 \bigl((1-Jz+\tau z^2)^2-4z^3\bigr).
 \label{eq:discriminant-factorization}
\end{equation}
After removing the square factor and setting $X=z^{-1}$, the residual
quadratic extension is birational to
\begin{equation}
 E_J:\qquad y^2=(X^2-JX+\tau)^2-4X.
 \label{eq:residual-curve}
\end{equation}
Its smooth completion is a genus-one curve whenever
\begin{equation}
 J^4\tau+J^3-8J^2\tau^2-36J\tau+16\tau^3-27\ne0.
 \label{eq:genus-one-nonsingularity}
\end{equation}
Indeed, the discriminant in $X$ of the quartic on the right of
\eqref{eq:residual-curve} is $256$ times the left-hand side of
\eqref{eq:genus-one-nonsingularity}.  The parameter $J$ is also the first
integral of the standard QRT dynamics for
$\beta_n=H_{n-1}H_{n+1}/H_n^2$.  The Somos--4 recurrence gives
$\beta_{n-1}\beta_n^2\beta_{n+1}=\beta_n+\tau$; hence the map
$(u,v)\mapsto(v,(v+\tau)/(uv^2))$ preserves
$uv+u^{-1}+v^{-1}+\tau(uv)^{-1}$.  Substitution of
$u=\lambda r/m^2$ and $v=m\eta/r^2$ yields \eqref{eq:J-invariant}.

\section{The two algebraic branches}

With
\[
 D(z)=B(z)^2-4A(z)C(z),
\]
the two algebraic branches may be represented by
\begin{equation}
 G(z)=\frac{-B(z)\mathbin{\pm}\sqrt{D(z)}}{2A(z)}.
 \label{eq:explicit}
\end{equation}
When the constant term of $A$ vanishes, the rationalized expressions
\[
 \frac{-2C(z)}{B(z)\mp\sqrt{D(z)}}
\]
can be used instead.  The sign in \eqref{eq:explicit} cannot be fixed
uniformly over the full parameter space.

Write
\[
 P(z,X)=A(z)X^2+B(z)X+C(z).
\]
The constant coefficients of \eqref{eq:A}--\eqref{eq:C} are
\[
 A_0=\lambda^2\eta m(m-1)r,\qquad
 B_0=-2\lambda^3\eta m(m-1)r,\qquad
 C_0=\lambda^4\eta m(m-1)r.
\]
Consequently,
\begin{equation}
 P(0,X)=A_0(X-\lambda)^2.
 \label{eq:double-root}
\end{equation}
Both algebraic branches therefore coalesce at $X=\lambda$ when $z=0$, and
\begin{equation}
 \partial_XP(0,\lambda)=2A_0\lambda+B_0=0.
 \label{eq:vanishing-derivative}
\end{equation}
The simple-root form of the formal implicit-function theorem is not
applicable at the origin.  When $m=1$, the polynomial has a factor $z$, and
the cancelled equation satisfies
\[
 \left.\frac{P(z,X)}z\right|_{z=0}
 =\lambda r^2(X-\lambda)
  ((\lambda r+\tau-\eta)X-\lambda(\lambda r+\tau)).
\]
Thus, generically, the second branch then has constant term
$\lambda(\lambda r+\tau)/(\lambda r+\tau-\eta)$ rather than
$\lambda$.  In either case the desired branch is selected by the paired
conditions $g_0=\lambda$ and $g_1=r$.  \begin{equation}
 G(z)=\lambda+zF(z),\qquad
 F(z)=\sum_{q\geq0}f_qz^q,\qquad f_q=g_{q+1}.
 \label{eq:desingularization}
\end{equation}
For
\[
 D_j=2\lambda A_j+B_j,\qquad
 E_j=\lambda^2A_j+\lambda B_j+C_j,
\]
where $C_3=0$, substitution into $P(z,G(z))=0$ gives
\[
 P(z,\lambda+zF)=\sum_{j=0}^3E_jz^j
 +z\left(\sum_{j=0}^3D_jz^j\right)F
 +z^2\left(\sum_{j=0}^3A_jz^j\right)F^2.
\]
Equation \eqref{eq:double-root} gives $E_0=D_0=0$.  An ordinary power-series
branch additionally requires
\begin{equation}
 E_1=0.
 \label{eq:compatibility}
\end{equation}
The next coefficient equation is
\begin{equation}
 A_0f_0^2+D_1f_0+E_2=0.
 \label{eq:first-coefficient}
\end{equation}
The quadratic in \eqref{eq:first-coefficient} factors as
\begin{equation}
 A_0f_0^2+D_1f_0+E_2
 =\lambda^2\eta mr(f_0-r)\bigl((m-1)f_0-mr\bigr).
 \label{eq:first-coefficient-factorization}
\end{equation}
Thus the prescribed value $f_0=g_1=r$ selects the required branch.  When
$m\ne1$, the other slope is $mr/(m-1)$.

\section{The two-step companion}
Let $G_{0;\ell,m,r,n,t}(x)=\sum_{q\geq0}g_qx^q$ denote the selected solution.
Its Hankel determinants satisfy
\begin{equation}
 (H_1,H_2,H_3,H_4)=(\ell,m,r,n),\qquad
 H_{q+4}H_q=H_{q+3}H_{q+1}+tH_{q+2}^2.
 \label{eq:parameter-meaning}
\end{equation}
The next two values are
\begin{align}
 H_5&=\frac{mn+tr^2}{\ell}, \label{eq:fifth-Hankel}\\
 H_6&=\frac{rH_5+tn^2}{m}
 =\frac{r(mn+tr^2)+t\ell n^2}{\ell m}. \label{eq:sixth-Hankel}
\end{align}
Remove the first coefficient and shift the remaining series to the origin:
\begin{equation}
 \widehat G(x)=\frac{G_{0;\ell,m,r,n,t}(x)-g_0}{x}
 =\sum_{q\geq0}g_{q+1}x^q.
 \label{eq:coefficient-tail}
\end{equation}
Its Hankel transform is
\begin{equation}
 \det(g_{i+j+1})_{0\leq i,j\leq q-1}
 =H_{q+2}\qquad(q\geq1).
 \label{eq:tail-shift-hankel-identity}
\end{equation}
It therefore begins with
\[
 r,\quad n,\quad
 \frac{mn+tr^2}{\ell},\quad
 \frac{r(mn+tr^2)+t\ell n^2}{\ell m},\ldots .
\]
This is exactly the two-step advanced orbit $(H_{q+2})_{q\geq1}$.  It therefore satisfies the same recurrence with parameter $t$, starting at positive order $q\geq1$.
Equivalently, it has the initial parameter tuple
\begin{equation}
 \left(r,n,
 \frac{mn+tr^2}{\ell},
 \frac{r(mn+tr^2)+t\ell n^2}{\ell m},t\right).
 \label{eq:two-step-parameter-shift}
\end{equation}
The assertion is an identity of Hankel transforms, not necessarily of the
underlying coefficient sequences.  

\section{A recurrence for the coefficients}

Using the notation of \eqref{eq:desingularization}, division of the expanded
quadratic equation by $z^2$ yields
\begin{equation}
 A(z)F(z)^2+(D_1+D_2z+D_3z^2)F(z)+E_2+E_3z=0.
 \label{eq:desingularized-quadratic}
\end{equation}
The branch condition is
\begin{equation}
 A_0r^2+D_1r+E_2=0,
 \label{eq:branch-selection}
\end{equation}
and the branch is simple precisely when
\begin{equation}
 2A_0r+D_1\ne0.
 \label{eq:nondegeneracy}
\end{equation}
Equivalently,
\begin{equation}
 (2A_0r+D_1)^2=D_1^2-4A_0E_2.
 \label{eq:branch-discriminant}
\end{equation}

Define
\begin{equation}
 U_q=\sum_{k=0}^{q}f_kf_{q-k},\qquad U_q=0\quad(q<0).
 \label{eq:Uq}
\end{equation}
Coefficient extraction from \eqref{eq:desingularized-quadratic} gives
\begin{equation}
 \sum_{j=0}^{3}A_jU_{q-j}
 +D_1f_q+D_2f_{q-1}+D_3f_{q-2}+E_{q+2}=0,
 \label{eq:convolution}
\end{equation}
where $f_q=0$ for $q<0$ and $E_j=0$ for $j>3$.  For $q\geq1$,
\begin{equation}
 U_q=2rf_q+\sum_{k=1}^{q-1}f_kf_{q-k}.
 \label{eq:Uq-split}
\end{equation}
Under condition \eqref{eq:nondegeneracy}, the coefficients are therefore generated
recursively by
\begin{equation}
 f_q=-\frac{
 A_0\displaystyle\sum_{k=1}^{q-1}f_kf_{q-k}
 +\displaystyle\sum_{j=1}^{3}A_jU_{q-j}
 +D_2f_{q-1}+D_3f_{q-2}+E_{q+2}}
 {2A_0r+D_1}.
 \label{eq:coefficient-recurrence}
\end{equation}
This quadratic convolution recurrence follows directly from the algebraic
generating-function equation.  On the selected branch,
\begin{equation}
 2A_0r+D_1=-\lambda^2\eta mr^2,
 \label{eq:selected-branch-denominator}
\end{equation}
so condition \eqref{eq:nondegeneracy} holds throughout the nondegenerate
S-fraction chart.  Because $F$ is algebraic of degree at most two, it is
D-finite.  Consequently, the same coefficient sequence also satisfies a
fixed-width linear recurrence with polynomial coefficients in the index.

\section{Base field and nondegeneracy}
The generic construction is over the rational-function field
$\mathbb Q(\lambda,m,r,\eta,\tau)$.  For scalar specializations, the initial
conditions require at least
\[
 \lambda mr\eta(m\eta+\tau r^2)\ne0,
\]
together with the nonvanishing of every later Hankel quantity that occurs as
a denominator.  Accordingly, every reconstruction and uniqueness statement
in this paper is understood within this nondegenerate S-fraction chart.

\section{Nonuniqueness of the coefficient sequence}
\label{sec:nonuniqueness}

The ordinary Hankel transform does not, in general, determine a coefficient
sequence.  If only $(H_n)$ is prescribed, the factorization
$c_{2j-1}c_{2j}$ visible in \eqref{eq:S-fraction-Hankel} leaves room for
variation in the individual Stieltjes coefficients.  For this reason a
five-parameter algebraic generating function should not be read as the unique
moment representation attached to an ordinary Somos--4 Hankel orbit.

The construction in this paper removes this ambiguity by imposing the
additional companion condition \eqref{eq:companion-Hankel-choice}.  In other
words, we prescribe two product sequences: the first is the Somos--4 sequence
$(H_n)$, and the second is the same sequence advanced by two determinant
indices.  Together with $c_0=\lambda$, these two prescriptions determine the
odd and even Stieltjes coefficients separately through
\eqref{eq:initial-S-coefficients} and \eqref{eq:S-coefficient-recurrence}.
Hence the algebraic generating function in Section~\ref{sec:algebraic-gf} is
the fully specified representative selected by those two Hankel-level
requirements.  This choice is already built into the initial Stieltjes
coefficients above; making it explicit here clarifies how the apparent degree
of freedom is eliminated.

Other moment sequences can still share the first ordinary Hankel transform
when the companion condition is not imposed.  For example, the binomial
transform, a standard Hankel-transform-preserving operation
\cite{Layman2001,French2007},
\begin{equation}
 g'_q=\sum_{k=0}^{q}\binom{q}{k}g_k,
 \label{eq:binomial-transform}
\end{equation}
has generating function
\[
 G'(z)=\frac{1}{1-z}G\!\left(\frac{z}{1-z}\right)
\]
and preserves the Hankel determinants.  Scalar and exponential rescalings of
coefficient sequences produce explicit normalization factors in the
determinants, while affine changes of the underlying variable give further
related moment functionals.  Equality of Hankel transforms must therefore not
be replaced by equality of generating functions without an independent
argument.

\section*{Disclosure on authoring and verification tools}
The text of this manuscript was revised and partially rewritten with the
assistance of AI-based language tools.  The purpose was to improve clarity and
place the mathematical material into proper \LaTeX while preserving the
technical content and the mathematical claims.  The author remains responsible
for the final text, all statements, and the correctness of the formulas.

AI-assisted workflows were also used to draft and test PARI/GP scripts for
verifying the identities and formulas appearing in this paper.  PARI/GP
merits explicit recognition as a powerful and highly effective tool for exact
symbolic and arithmetic verification, and it played a useful role in checking
the computations underlying the present work.

\end{document}